\documentclass[reqno,11pt]{article}
\usepackage{a4wide,xcolor,eucal,enumerate,mathrsfs}
\usepackage[normalem]{ulem}
\usepackage[pdfborder={0 0 0}]{hyperref}  

\usepackage{amsmath,amssymb,epsfig,amsthm}
\usepackage[latin1]{inputenc}

\numberwithin{equation}{section}

\newtheorem{theorem}{Theorem}[section]

\newtheorem{lemma}[theorem]{Lemma}

\newtheorem{definition}[theorem]{Definition}

\newcommand{\ppi}{{\mbox{\boldmath$\pi$}}}
\newcommand{\ggamma}{{\mbox{\boldmath$\gamma$}}}
\newcommand{\sggamma}{{\mbox{\scriptsize\boldmath$\gamma$}}}

\newcommand{\supp}{\mathop{\rm supp}\nolimits}
\renewcommand{\d}{{\mathrm d}}

\newcommand{\R}{\mathbb{R}}
\newcommand{\V}{\mathbb{V}}
\newcommand{\dom}{\rm D}

\newcommand{\LL}{\mathscr{L}}

\newcommand{\mm}{\mathfrak m}

\newcommand{\sfd}{{\sf d}}

\newcommand{\ff}{{\mbox{\boldmath$f$}}}
\newcommand{\prob}[1]{\mathscr P(#1)}
\newcommand{\probt}[1]{\mathscr P_2(#1)}

\newcommand{\geo}{{\rm{Geo}}}
\newcommand{\e}{{\rm{e}}}

\newcommand{\CD}{{\sf CD}}
\newcommand{\RCD}{{\sf RCD}}
\newcommand{\BE}{{\sf BE}}
\newcommand{\Ch}{{\sf Ch}}
\renewcommand{\H}{{\sf H}}

\newcommand{\LIP}{{\sf LIP}}
\newcommand{\DG}{\dom_\V(\Delta)\cap \LIP(X)}

\newcommand{\cI}{{\mathcal I}}
\newcommand{\rmD}{{\mathrm D}}

\title{Gaussian-type Isoperimetric  Inequalities in $\RCD(K,\infty)$ probability spaces for positive $K$}
\begin{document}
\author{Luigi Ambrosio
\thanks{Scuola Normale Superiore di Pisa (Italy) \emph{Email}: \textsf{luigi.ambrosio@sns.it}}
      \and
   Andrea Mondino   \thanks{MSRI-Berkeley\&Universit\"at Z\"urich, Institut f\"ur Mathematik. \emph{Email}: \textsf{andrea.mondino@math.uzh.ch}}
  }

\maketitle

\begin{abstract} In this paper we adapt the well-estabilished $\Gamma$-calculus techniques to the context of $\RCD(K,\infty)$
spaces, proving Bobkov's local isoperimetric inequality \cite{BL96},\cite{Bobkov97} and, when $K$ is positive, the Gaussian isoperimetric inequality in this class of spaces.
The proof relies on the measure-valued $\Gamma_2$ operator introduced by Savar\'e in \cite{Savare2013}.
\end{abstract}


%\textit{Mathematics Subject Classification:} 51F99-53B99 ??.
%\tableofcontents

\section{Introduction}
The goal of the paper is to extend to $\RCD(K,\infty)$ spaces, a class of non-smooth ``Riemannian'' and possibly infinite-dimensional 
metric measure spaces with lower Ricci curvature bounds, the L\'evy-Gromov Gaussian isoperimetric inequality. Our proof is strongly inspired by the arguments of 
Bakry-Ledoux \cite{BL96} (see also the recent monograph \cite{BGL14}), where the authors gave a proof of such an inequality 
for infinite dimensional diffusion operators via $\Gamma$-calculus. This calculus, also according to the presentation in \cite{BGL14}, 
relies on the assumption of the existence of a dense algebra of functions stable under the $\Gamma$ operator, an assumption that
seems to be problematic in the category of $\RCD(K,\infty)$ spaces. Savar\'e realized in
\cite{Savare2013} that, by looking at the iterated $\Gamma$ operator $\Gamma_2$ as a measure-valued operator, the assumption on the algebra
can be relaxed, and that most techniques of $\Gamma$-calculus still apply (heuristically, because the negative part of the tensor
$\Gamma_2$ is still absolutely continuous w.r.t. $\mm$); this point of view has also been exploited in \cite{Gigli} for the development
of Hessian estimates and a full theory of covariant derivatives. In this paper we follow the same line of thought to prove Bobkov's 
inequality and the L\'evy-Gromov Gaussian isoperimetric inequality in sharp form. On the technical side, while the computations 
are essentially the same (see in particular Chapter 8.5.2 of \cite{BGL14}), we have tried to emphasize the technical issues arising from the not-so-smooth
context of $\RCD(K,\infty)$ spaces,  in particular the lack of smoothness of the function $\sqrt{\cI^2(g)+c_\alpha(t)\Gamma(g)}$ 
appearing in the proof of Bobkov's inequality.

One of the motivations for writing this note comes from \cite{AmbrosioHonda}, where the isoperimetric result 
of this paper is used to prove
  new compactness results in variable Sobolev spaces $H^{1,p}(X,\sfd_i,\mm_i)$ relative to
$\RCD(K,\infty)$ metric measure structures $(X,\sfd_i,\mm_i)$, in the whole range of exponents 
$p\in [1,\infty)$ ($p=1$ corresponds to $BV(X,\sfd_i,\mm_i)$).

Let us mention that the isoperimetric inequality for Gaussian measures is originally  due to
C. Borell \cite{Borell} and V. Sudakov - B. Cirel'son \cite{SuCi}. The functional form of
the Gaussian isoperimetric inequality  presented  in Theorem \ref{thm:Bob} was introduced
by S. Bobkov \cite{Bobkov97} who established it first on the two-point space and then in the
limit for the Gaussian measure by the central limit Theorem.   On the basis of the Bobkov inequality,
the local inequalities of Theorem \ref{thm:LocBob}  were established by D. Bakry and M. Ledoux \cite{BL96} in the context of smooth infinite dimensional spaces with lower Ricci curvature bounds.
Theorem \ref{thm:IsopIneq}  may be considered as the \emph{non-smooth, infinite-dimensional }extension of a celebrated result
of P. L\'evy and M. Gromov \cite{Gro} comparing the isoperimetric profile
of a Riemannian manifold with a strictly positive lower bound on the Ricci curvature
to that of the sphere with the same (constant) curvature and dimension. Let us mention that the dimensional L\'evy-Gromov inequality has been recently extended to non-smooth spaces with lower Ricci curvature bounds  in \cite{CM}. 

\medskip
\noindent {\bf Acknowledgment:} 
 Part  of the work has been  developed while A. M. was lecturer at the  Institut f\"ur Mathematik at the  Universit\"at Z\"urich and part when he was in residence at the Mathematical Science Research Institute in Berkeley, California, during the spring  semester 2016 and was supported by the National Science Foundation under the Grant No. DMS-1440140. He wishes to express his gratitude to both the institutes for the stimulating atmosphere and the excellent working conditions.   

\section{Preliminaries}\label{sec:preli}
Throughout the paper $(X,\sfd,\mm)$ will be a metric measure space, m.m.s. for short, i.e. $(X,\sfd)$ is a complete and separable metric space and $\mm$ is a nonnegative Borel measure. Even if some of the statements of this paper hold in case $\mm$ is a sigma finite measure, for simplicity we will always assume $\mm(X)=1$ and $\supp(\mm)=X$.  We shall denote by $\LIP(X)$ the space
of Lipschitz functions,  by $\prob X$ the space of Borel probability measures on the complete and separable metric space $(X,\sfd)$.
We also denote by 
\[
\probt X :=\Big\{\mu\in\prob X\ :\ \int_X\sfd^2(x_0,x)\,\d\mu(x)<\infty\,\,\,\text{for some (and hence all) $x_0\in X$}\Big\}.
\] 
the subspace consisting of all the probability measures with finite second moment.  The \emph{relative entropy functional} ${\rm Ent}_\mm:\probt{X}\to (-\infty,\infty]$ 
with respect to a probability measure $\mm$ is defined by
\begin{equation}\label{eq:defRelentropy}
{\rm Ent}_\mm(\mu):=
\begin{cases}
\int_X f\log f\,\d\mm& \text{if $\mu=f\mm$;}\\
\infty & \text{otherwise.}
\end{cases}
\end{equation}
Note that, by Jensen's inequality, this functional is nonnegative. 
\\
A curve $\gamma:[0,1]\to X$ is a \emph{geodesic} if
\begin{equation}\label{defgeo}
\sfd(\gamma_s,\gamma_t)=|t-s|\sfd(\gamma_0,\gamma_1)\qquad\forall s,\,t\in
[0,1].
\end{equation}
We will denote by $\geo(X)$ the space of all constant speed
geodesics $\gamma:[0,1]\to X$, namely $\gamma\in\geo(X)$ if
\eqref{defgeo} holds.

\subsection{Lower Ricci curvature bounds}\label{Subsec:CB} 

In the sequel we briefly recall those basic definitions and properties of spaces with lower Ricci curvature bounds that we will need later on.

For $\mu_0,\mu_1 \in \probt X$ the quadratic transportation distance $W_2(\mu_0,\mu_1)$ is defined by
\begin{equation}\label{eq:Wdef}
  W_2^2(\mu_0,\mu_1) = \inf_\sggamma \int_X \sfd^2(x,y) \,\d\ggamma(x,y),
\end{equation}
where the infimum is taken over all $\ggamma \in \prob{X \times X}$ with $\mu_0$ and $\mu_1$ as the first and the second marginal.
It turns out that any geodesic $(\mu_t) \in \geo(\probt X)$ can be lifted to a measure $\ppi \in \prob{\geo(X)}$, in such a way that $(\e_t)_\#\ppi = \mu_t$ for all $t \in [0,1]$. 

We turn to the formulation of the $\CD(K,\infty)$ condition, coming from the seminal works of Lott-Villani \cite{Lott-Villani09} and Sturm \cite{Sturm06I}. 
\begin{definition}[$\CD(K,\infty)$ condition]
Given $K \in \R$,  we say that a m.m.s.  $(X,\sfd,\mm)$
 is a $\CD(K,\infty)$-space if for any two measures  $\mu_0,\,\mu_1 \in D({\rm Ent}_\mm)$ there  exists
a geodesic $(\mu_t) \in \geo(\probt X)$ which 
satisfies the convexity inequality
\begin{equation}\label{eq:CDdef}
{\rm Ent}_\mm (\mu_t) \le (1-t) {\rm Ent}_\mm (\mu_0) + t {\rm Ent}_\mm (\mu_1)
                    - \frac{K}{2}t(1-t)W_2^2(\mu_0,\mu_1), \quad \text{for all $t \in [0,1]$}.
\end{equation}
 \end{definition}
 
Notice that the above-mentioned lifting property shows that the $\CD(K,\infty)$ condition implies that the class 
 $\geo(X)$ is sufficiently rich to provide $\ppi \in \prob{\geo(X)}$ connecting any two measures with finite entropy. Using this,
 one can prove that $X=\supp\mm$ is length, i.e. $\sfd(x,y)$ is the infimum of the length of the curves connecting $x$ to $y$
 for all $x,\,y\in\supp\mm$.
  
 Two crucial properties of the $\CD(K,\infty)$-condition are the consistency with the smooth counterpart and the stability under convergence. More 
 precisely, it was shown both in  \cite{Lott-Villani09} and \cite{Sturm06I} that a smooth Riemannian manifold satisfies the  $\CD(K,\infty)$-condition if and only if it has Ricci curvature bounded below by $K$. Regarding the stability, it was proved by Lott-Villani \cite{Lott-Villani09} that the $\CD(K,\infty)$-condition is stable under pointed  measured Gromov-Hausdorff convergence of \emph{proper} pointed metric measure spaces; Sturm \cite{Sturm06I} instead introduced a new distance, denoted with  ${\mathbb D}$, between normalized metric measure spaces and proved the stability of the $\CD(K,\infty)$-condition under such a convergence. The stability in the  general case of non-proper and non-normalized spaces was established in \cite{GigliMondinoSavare} where a new notion of convergence, called pointed measured Gromov convergence, was introduced (let us mention that such a convergence coincides with the measured Gromov Hausdorff one in case of uniformly doubling spaces and with the ${\mathbb D}$-convergence in case of normalized spaces) and the stability of the $\CD(K,\infty)$-condition under such a convergence was established.

\subsection{The Cheeger energy,  its gradient flow and the $\RCD(K,\infty)$-condition}
First of all recall that the  metric slope of a Lipschitz function $f \in \LIP(X)$ is defined by
\begin{equation}\label{eq:defDF}
{\rm lip\,}f (x):=\limsup_{y\to x} \frac{|f(x)-f(y)|}{\sfd(x,y)}
\end{equation}
(with the convention ${\rm lip\,}f(x)=0$ if $x$ is an isolated point).
The Cheeger energy (introduced in \cite{Cheeger00} and further studied in \cite{Ambrosio-Gigli-Savare11}) is defined as the $L^{2}$-lower semicontinuous envelope of the functional $f \mapsto \frac{1}{2} \int_{X} {\rm lip}^2f \, \d \mm$, i.e.:
\begin{equation}
\Ch_{\mm}(f):=\inf \left\{ \liminf_{n\to \infty}  \frac{1}{2} \int_{X} {\rm lip}^2f_n\, \d \mm \, : \, f_n\in \LIP(X), \; f_{n}\to f \text{ in }L^{2}(X,\mm) \right \}.
\end{equation}
If $\Ch_{\mm}(f)<\infty$ it was proved in \cite{Cheeger00,Ambrosio-Gigli-Savare11} that the set
$$
G(f):= \left\{g \in L^{2}(X,\mm)\,: \,  \exists \, f_{n} \in \LIP(X), \, f_n\to f, \, {\rm lip\,}f_n \rightharpoonup h\geq g  \text{ in } L^{2}(X,\mm) \right\}
$$
is closed and convex, therefore it
admits a unique element of minimal norm called  \emph{minimal weak upper gradient} and denoted by $|Df|_{w}$.   The Cheeger energy can be then  represented by integration as
$$\Ch_{\mm}(f):=\frac{1}{2} \int_{X} |Df|_{w}^{2} \, \d \mm. $$
It is not difficult to see that $\Ch_{\mm}$ is a $2$-homogeneous, lower semi-continuous, convex functional on $L^{2}(X,\mm)$, whose proper domain ${\dom}(\Ch_{\mm}):=\{f \in L^{2}(X,\mm)\,:\, \Ch_{\mm}(f)<\infty\}$ is a dense linear subspace of $L^{2}(X,\mm)$. It then admits an $L^{2}$-gradient flow which is a continuous semigroup of contractions $(\H_{t})_{t\geq 0}$ in $L^{2}(X,\mm)$, whose continuous trajectories $t \mapsto \H_{t} f$, for $f \in L^{2}(X,\mm)$, 
are locally Lipschitz  curves from $(0,\infty)$ with values into  $L^{2}(X,\mm)$.  We can now define the $\RCD(K,\infty)$-condition introduced in \cite{Ambrosio-Gigli-Savare11b} (see also \cite{AmbrosioGigliMondinoRajala} for the present simplified axiomatization and for the generalization to $\sigma$-finite reference measures).

\begin{definition}[$\RCD(K,\infty)$-space]
Let  $(X,\sfd,\mm)$ be a probability metric measure space. 
We say that $(X,\sfd,\mm)$ is an $\RCD(K,\infty)$-space if it satisfies the $\CD(K,\infty)$-condition, and moreover  
Cheeger's energy is quadratic, i.e. it satisfies the parallelogram identity
 \begin{equation}\label{eq:ChQuad}
 \Ch_{\mm}(f+g)+\Ch_{\mm}(f-g)=2\Ch_{\mm}(f) + 2 \Ch_{\mm}(g), \quad \forall f,\,g \in \dom(\Ch_{\mm}).
 \end{equation}
\end{definition}

If $(X,\sfd,\mm)$ is an $\RCD(K,\infty)$-space then the Cheeger energy induces 
the Dirichlet form ${\mathcal E}(f):=2 \Ch_\mm(f)$ which is strongly local and admits the Carr\'e du Champ
\begin{equation}\label{eq:defGamma}
\Gamma(f):=|Df|_{w}^{2}, \quad \text{for every }f \in {\dom}(\Ch_{\mm}).
\end{equation}
Moreover the sub-differential $\partial \Ch_{\mm}$ is single-valued and coincides with the linear generator $\Delta$ of the heat flow semi-group $(\H_{t})_{t\geq 0}$ defined above. 

The $\RCD(K,\infty)$-condition was introduced in order to single out in a more precise way, compared to the $\CD$ theory, non-smooth ``Riemannian'' structures; indeed, while on the one hand it is possible to give examples of smooth Finsler manifolds satisfying the $\CD(K,\infty)$-condition, on the other hand it is not difficult to see that if a Finsler manifold satisfies \eqref{eq:ChQuad} then it is actually Riemannian. Let us finally mention that the stability of the  $\RCD(K,\infty)$ condition was proved in \cite{Ambrosio-Gigli-Savare11b} for normalized metric measure spaces with respect to the  ${\mathbb D}$-convergence of Sturm, and the general case of  pointed-Gromov convergence of spaces endowed with $\sigma$-finite reference measures  was settled in \cite{GigliMondinoSavare}.

\subsection{The $\BE(K,\infty)$-condition and its  self-improvement} \label{SS:GammaCalc}

We set by definition $\V:={\dom}(\Ch_{\mm})$ and we endow such space with the norm $$\|f\|_{\V}:=\sqrt{\|f\|^2_{L^{2}(X,\mm)}+2 \Ch_{\mm}(f)}.$$
 If $(X,\sfd,\mm)$ is an $\RCD(K,\infty)$-space then $\V$ is a Hilbert space, while it is not difficult to see that in a general m.m.s. $\V$ is a Banach space (indeed \eqref{eq:ChQuad} is equivalent to require that $\V$ is Hilbert).
Moreover, still under the $\RCD(K,\infty)$ assumption, 
the very construction of $\Ch_{\mm}$ and the fact that $\Ch_\mm$ is quadratic imply that $\LIP(X)\cap L^{\infty}(X,\mm)$ is dense in $\V$ (\cite[Proposition 4.10]{Ambrosio-Gigli-Savare11b}).
 
The Laplace operator, introduced above as the infinitesimal generator of the Heat semi-group, can be defined also via integration by parts as the operator  $-\Delta:\V\to \V'$ given by
\begin{equation}
  \label{eq:defDelta}
  \int_X  (-\Delta f) \; g \, \d \mm :=\int_X \Gamma(f,g) \d \mm, \quad\forall f,\,g\in \V,
\end{equation}
and it  is an unbounded self-adjoint non-negative operator with domain ${\dom}(\Delta)$.
\\ Let us also introduce the functional spaces 
\begin{equation}\label{eq:defDVDelta}
{\dom}_\V(\Delta)=\big\{f\in \V\,:\,
  \Delta f\in \V\big\}
 \end{equation} 
 and 
 \begin{equation}\label{eq:defDLinfDelta}
 {\dom}_{L^{\infty}}(\Delta)=\big\{f\in \V\cap L^{\infty}(X,\mm)\,:\,
  \Delta f\in L^{\infty}(X,\mm)\big\}.
\end{equation}
If $(X,\sfd,\mm)$ is an $\RCD(K,\infty)$-space, the multilinear form $\Gamma_{2}:{\dom}_{\V}(\Delta)\times {\dom}_{\V}(\Delta)\times {\dom}_{L^{\infty}}(\Delta) \to \R$ is defined by
$$\Gamma_{2}[f,g;\varphi]:=\frac{1}{2}\int_{X}\left[ \Gamma(f,g) \, \Delta \varphi- \left( \Gamma(f,\Delta g)+\Gamma(g, \Delta f)\right)  \varphi \right] \, \d \mm.$$
When $f=g$ we set
$$\Gamma_{2}[f;\varphi]:=\Gamma_{2}[f,f;\varphi]:=\int_{X}\left(\frac{1}{2} \Gamma(f) \, \Delta \varphi-  \Gamma(f,\Delta g)  \varphi \right) \, \d \mm, $$
so that
$$
\Gamma_{2}[f,g;\varphi]=\frac{1}{4}\Gamma_{2}[f+g;\varphi]-\frac{1}{4}\Gamma_{2}[f-g;\varphi].
$$
The $\Gamma_{2}$ operator provides a weak version of the Bakry-\'Emery condition \cite{BE83} (see also  \cite{B94} and the recent monograph \cite{BGL14}). More precisely it was proved in \cite{Ambrosio-Gigli-Savare11b} 
(see also \cite{AmbrosioGigliMondinoRajala} for the $\sigma$-finite case) that $\RCD(K,\infty)$ implies 
the Bakry-\'Emery condition $\BE(K,\infty)$:

\begin{theorem}
Let $(X,\sfd,\mm)$ be an $\RCD(K,\infty)$-space, then $\BE(K,\infty)$-condition holds, i.e.
\begin{equation}\label{eq:BEKinf}
\Gamma_{2}[f;\varphi]\geq K \int_{X}\Gamma(f) \, \varphi \, \d \mm, \quad \text{for every } (f,\varphi)\in {\dom}_{\V}(\Delta)\times {\dom}_{L^{\infty}}(\Delta), \; \varphi \geq 0 \; \mm\text{-a.e. }. 
\end{equation}
\end{theorem}
Under natural regularity assumptions on the distance canonically associated to the Dirichlet form it is also to prove the converse 
implication, i.e. from $\BE(K,\infty)$ to $\RCD(K,\infty)$, see \cite{AGSBaEm} (see also \cite{ErbarKuwadaSturm} and \cite{AmbrosioMondinoSavareNLDiff} for the dimensional case).

A crucial property of the $\BE(K,\infty)$ condition that will be used later in the paper, is the self-improvement established by Savar\'e \cite[\S\,3]{Savare2013} in the present setting. 
 
\begin{theorem}\label{thm:Gamma2Impr}
  Let $(X,\sfd,\mm)$ be an $\RCD(K,\infty)$-space.
  \begin{enumerate}[\rm (1)]
  \item 
    For every
    $f \in {\dom}_\V(\Delta)\cap \LIP(X)$
    we have $\Gamma(f)\in \V$ with
  \begin{equation}
    \label{eq:G3}
    \Ch_{\mm}(\Gamma(f)) \leq  -\int_X \Big(
    2K\Gamma(f)^2+ 2\Gamma(f)\Gamma(f, \Delta f) \Big)\,\d\mm.
  \end{equation}
  \item
  ${\dom}_\V(\Delta)\cap \LIP(X)$ is an algebra (i.e. closed w.r.t.~pointwise
  multiplication)
  and, more generally,
  if $\ff=(f_i)_{i=1}^n\in {(\DG)}^n$ then $\Phi(\ff)\in \DG$ 
  for every smooth function $\Phi:\R^n\to \R$ with $\Phi(0)=0$.
  \item If $f\in {\dom}_\V(\Delta)\cap \LIP(X)$, then the linear functional
  {$$
  \varphi\in\LIP(X)\cap L^\infty(X,\mm)\mapsto -\int_X\Gamma(\varphi,\Gamma(f))\,\d\mm
  $$}
  can be represented by a signed Borel measure  $\Delta^\star \Gamma(f)$ which
  can be extended to a unique element in $\V'$ (by integration of the $\Ch_\mm$-quasicontinuous
  representative) and, by defining
  \begin{equation}
    \label{eq:50}
    \Gamma^\star_{2,K} [f]:=\frac 12 \Delta^\star \Gamma(f)-\Big(\Gamma(f,{\Delta f})+K\Gamma(f)\Big)\mm,
  \end{equation}
  the measure $\Gamma^\star_{2,K}[f]$ is nonnegative and satisfies
   \begin{equation}
  \label{eq:81}
  \Gamma_{2,K}^\star[f](X)
  \le \int_X \Big(\big(\Delta f\big)^2-K\Gamma(f)\Big)\,\d\mm.
  \end{equation} 
    \item 
    There exists a continuous, symmetric and bilinear map 
     $$\gamma_{2,K}:\big({\dom}_\V(\Delta)\cap \LIP(X)\big)^2\to L^1(X,\mm)$$ 
     such that for every $f\in {\dom}_\V(\Delta)\cap \LIP(X)$
     one has
     \begin{equation}
       \label{eq:2}
       \Gamma_{2,K}^\star[f]=\gamma_{2,K} [f,f]\mm+\Gamma_{2,K}^\perp[f],
       \quad\text{with}\quad \Gamma^\perp_{2,K}[f]\geq 0,\quad \Gamma_{2,K}^\perp[f]\perp \mm.
     \end{equation}
         Setting $\gamma_{2,K}[f]:=\gamma_{2,K}[f,f] \geq 0$, one has 
     for every $f\in \DG$
  \begin{align}
     \label{eq:57}
  \Gamma\left(\Gamma (f)\right)&\le 4 \gamma_{2,K}  [f]\,\Gamma(f)\quad
  \text{$\mm$-a.e.~in $X$.}
\end{align}
\end{enumerate}
\end{theorem}

Notice that the measures $\Gamma_{2,K}^\star [f]$, $K\in \R$, just
differ by a multiple of $\Gamma(f)\mm$, so the (non-negative) singular part in 
the Lebesgue decomposition \eqref{eq:2} is independent of $K$. 
For more results about functional inequalities and regularity in $\BE(K,\infty)$-spaces the interested reader is 
referred to \cite{AGSBaEm} and  \cite{AmbrosioMondinoSavare}.

\subsection{Improved regularity of the heat flow in $\RCD(K,\infty)$-spaces}\label{subsec:ImprovedRegHeat}
{In $\RCD(K,\infty)$-spaces the heat flow $\H_t$ can be viewed in two conceptually different, but consistent, ways: the first one, as in the
theory of Dirichlet form and in the more general theory of gradient flows of lower semicontinuous functionals in Hilbert spaces,
is based on $\H_t$ as the $L^2(X,\mm)$ gradient flow of $\Ch_\mm$; the second way looks at $\H_t$ as the
gradient flow ${\cal H}_t$ of the Shannon entropy functional  in the Wasserstein space. It has been proved in full generality 
in \cite{Ambrosio-Gigli-Savare11} that the two points of view coincide in the common domain of probability densities
in $L^1(X,\mm)$. In addition, using this identification, in}
\cite{Ambrosio-Gigli-Savare11b}  several regularity properties of $\H_t$, besides self-adjointness, Markov, mass-preserving,
have been deduced. We recall some of them. 
\\When $f \in  L^\infty(X,\mm)$, for all $t>0$ the function $\H_t f$ has a continuous representative, denoted by $\tilde{\H}_t f$, 
which is defined as follows (see Theorem 6.1 in \cite{Ambrosio-Gigli-Savare11b})
\begin{equation}\label{eq:ContRepHt}
\tilde{\H}_t f:=\int_X f \, \d {\cal H}_t(\delta_x).
\end{equation}
Moreover, for each $f \in L^\infty(X,\mm)$ the map $(t,x)\mapsto \tilde{\H}_t f(x)$ belongs to $C_b((0,\infty)\times X)$.
 Finally, for all $f\in L^\infty(X,\mm)$ with $\Gamma(f)\in L^\infty(X,\mm)$
 the classical Bakry-\'Emery gradient estimate holds even in the pointwise form (see Theorem 6.2 in \cite{Ambrosio-Gigli-Savare11b})
\begin{equation}\label{eq:BEflow}
{\rm lip\,}\tilde{\H}_t f \leq e^{-2Kt} \tilde{\H}_t( \Gamma(f)) \qquad\text{in $X=\supp\mm$, for all $t>0$,}
\end{equation}
while if we drop the boundedness assumption on $\Gamma(f)$ one has (with ${\rm Lip}$ denoting the global
Lipschitz constant)
\begin{equation}\label{eq:BEflow2}
{\rm Lip\,}(\tilde{\H}_t f )\leq \frac{1}{\sqrt{2{\sf I}_{2K}(t)}}\|f\|_{L^\infty(X,\mm)}\quad\text{with ${\sf I}_{2K}(t)=\frac{e^{2Kt}-1}{2K}$}.
\end{equation}
For functions $f\in\V$, instead, one has the traditional form of $\Gamma$-calculus
\begin{equation}\label{eq:BEflow1}
\Gamma(\H_t f) \leq e^{-2Kt} \H_t( \Gamma(f)) \qquad\text{$\mm$-a.e., for all $t\geq 0$.}
\end{equation}
From now on, when $f\in L^\infty(X,\mm)$, we will always choose $\tilde{\H}_{t} f $ as representative of $\H_{t}f$ and 
we will simply write $\H_{t}f$.

{An immediate consequence of \eqref{eq:BEflow}, \eqref{eq:BEflow1} and of the length property of $X=\supp\mm$ is that the Dirichlet form
induced by $\Ch_\mm$ is irreducible, i.e. $\Ch_\mm(f)=0$ implies that $f$ is equivalent to a constant (this holds because
all functions $\tilde\H_t f$, $t>0$, have null slope in $X$ and therefore are $\mm$-equivalent to a
constant). Then, since irreducibility implies ergodicity (see for instance \cite[Section~3.8]{BGL14}) we get
\begin{equation}\label{eq:ergod}
\lim_{t\to +\infty} \H_{t}f=\int_X f \, \d \mm\qquad \text{in $L^2(X,\mm)$, for all $f\in L^2(X,\mm)$.} 
\end{equation}
}

The following lemma will be useful in the proof of the Bobkov inequality.

\begin{lemma}\label{lem:regHF}
Let $(X,\sfd,\mm)$ be an $\RCD(K,\infty)$ probability space. Then, for every   fixed $f\in L^2(X,\mm)$, the map $\H_{(\cdot)}f: (0,T]\to \V$ is locally
Lipschitz, continuous up to $t=0$ if $f\in\V$, with also $t\mapsto\Gamma(\H_t f)$ locally Lipschitz in $(0,T]$  as a $L^1(X,\mm)$-valued map, and
\begin{equation}\label{eq:der_Gamma}
\lim_{s\to t}\frac{\Gamma(\H_s t)-\Gamma(\H_t f)}{s-t}=\Gamma(\Delta\H_t f,f)
\qquad\text{in $L^1(X,\mm)$ for $\LL^1$-a.e. $t>0$.}
\end{equation}
Moreover, if  $f \in L^{\infty}(X,\mm)$,  then for every $t>0$ we have  $\H_{t}f \in  {\dom}_\V(\Delta)\cap \LIP(X)$. 
\end{lemma}

\begin{proof} The proof of the continuity of $\H_tf$ and the $\Delta$-regularity of $\H_tf$
follow immediately by the regularization estimates provided by the theory of 
gradient flows in Hilbert spaces (see for instance \cite{Ambrosio-Gigli-Savare08}), namely
$$
\Ch_\mm(\H_tf)\leq\inf_{g\in\V}\frac{\|g-f\|_{L^2(X,\mm)}^2}{2t},
$$
$$
\|\Delta\H_t f\|_2^2\leq\inf_{g\in D(\Delta)}\frac{\|g-f\|_{L^2(X,\mm)}^2}{t^2},
$$
using also the semigroup property and the commutation $\H_t\circ\Delta=\Delta\circ\H_t$ for $t>0$. Formula \eqref{eq:der_Gamma}
is a simple consequence of $\Gamma(\H_s f)-\Gamma(\H_t f)=\Gamma(\H_s f-\H_t f,\H_sf+\H_t f)$ and of the differentiability of
$t\mapsto \H_tf$ as a $\V$-valued map.

Finally, if $f\in L^\infty(X,\mm)$
the Lipschitz regularity of $\H_t f$ follows directly from \eqref{eq:BEflow2}.
\end{proof}

Thanks to Lemma \ref{lem:regHF} we will be able to apply Theorem \ref{thm:Gamma2Impr} to $\H_{t} f$ for $f\in L^{\infty}(X,\mm), t>0$.

\subsection{The Gaussian isoperimetric profile}
Let $H(r):=\frac{1}{\sqrt{2\pi}}\int_{-\infty}^{r} e^{-x^{2}/2} dx$, $r\in \R$, 
be the distribution function of the standard Gaussian measure on the real line and set $h:=H'$ 
to be its density with respect to the Lebesgue measure $\LL^{1}$. Then
\begin{equation}\label{eq:defcI}
\cI:=h\circ H^{-1}:[0,1]\to \left[0, \frac{1}{\sqrt{2 \pi}} \right]
\end{equation}
defines the \emph{Gaussian isoperimetric profile} (in any dimension). Note that the function $\cI$ is concave, 
continuous, symmetric with respect to $1/2$, it satisfies $\cI(0)=\cI(1)=0$ and the fundamental differential equation $\cI \, \cI''=-1$.

\section{Local Bobkov inequality in $\RCD(K,\infty)$-spaces, $K\in \R$}
The goal of this section is to prove the following result.
\begin{theorem}[Local Bobkov Inequality]\label{thm:LocBob}
Let $(X,\sfd,\mm)$ be an $\RCD(K,\infty)$ probability space for some $K\in \R$. Then, for every $f \in  \LIP(X)$ 
with values in $[0,1]$, every $\alpha\geq 0$ and every $t\geq 0$, it holds
\begin{equation}\label{eq:LocBob}
\sqrt{\cI^{2}(\H_{t}f)+ \alpha \Gamma(\H_{t}f)} \leq \H_{t} \left( \sqrt{\cI^{2}(f)+c_{\alpha}(t)  \Gamma(f)} \right) , \quad \mm\text{-a.e.},
\end{equation}
where, for $t\geq 0$, we have set
\begin{equation}\label{eq:defcalpha}
c_{\alpha}(t):=\frac{1- e^{-2Kt} }{K}+\alpha e^{-2Kt},\; \text{ if $K\neq 0$}; \quad c_{\alpha}(t):=2 t+\alpha ,\;   \text{ if $K= 0$}.
\end{equation}
\end{theorem}
In the next lemma we isolate a key computation for proving  Theorem \ref{thm:LocBob}.

\begin{lemma}\label{lem:compLocBobkov}
Let $(X,\sfd,\mm)$ be an $\RCD(K,\infty)$ probability space for some $K \in \R$. Let $\Psi(t,u,v):\R^3\to \R$ be a function of class 
$C^4$ with $\partial_v\Psi\geq 0$,  fix $T>0$,  $f\in \LIP(X)\cap L^{\infty}(X,\mm)$, 
$\varphi\in L^{\infty}(X,\mm)$ with $\varphi\geq 0$ $\mm$-a.e., and set 
\begin{equation}\label{eq:defPhi}
\Phi(t):=\int_{X} \H_{t}\left(\Psi(t, \H_{T-t}(f), \Gamma(\H_{T-t}(f)) \right)\, \varphi \, \d \mm, \quad \forall t \in [0,T].
\end{equation}
Then  $\Phi$ is continuous in $[0,T]$, locally Lipschitz in $(0,T)$, and for $\LL^{1}$-a.e. $t_0 \in (0,T)$ it holds
\begin{equation}\label{eq:dtPhi}
\frac{d}{dt}_{|t=t_{0}} \Phi(t) \geq \int_{X} \zeta  \; \H_{t_{0}} \varphi \, \d \mm \;,
\end{equation}
where,  denoting $g:=H_{T-t_{0}}f$ and writing  simply $\Psi$ in place of $\Psi(t_{0}, g , \Gamma(g))$,   
 we have set 
 \begin{equation}\label{eq:defzeta}
\zeta:=   \partial_t\Psi+  \partial^2_u\Psi \, \Gamma(g)+2\, \partial_u \partial_v \Psi \, \Gamma(g, \Gamma(g))
+\partial^2_v\Psi\, \Gamma(\Gamma(g))  + 2K \, \partial_v\Psi  \,\Gamma(g) + 2\, \partial_v \Psi\; \gamma_{2,K}[g]. 
\end{equation}
\end{lemma}

\begin{proof}
\textbf{Step 1}.
{$\Phi:[0,T]\to \R$ is locally Lipschitz in $(0,T)$, continuous up to $t=0,T$, and Leibniz rule.}
\\To this aim recall that, in virtue of Lemma \ref{lem:regHF}, for every fixed  $f \in L^2(X,\mm)$  the heat flow $\H_{(\cdot)}f:(0,T]\to \V$ is locally Lipschitz. 
Moreover, since by assumption $f \in \LIP(X)\cap L^{\infty}(X,\mm)$,     on the one hand by maximum principle it holds 
$\|\H_{t} f\|_{L^\infty(X,\mm)}\leq \|f\|_{L^\infty(X,\mm)}$ and on the other hand, in virtue of   \eqref{eq:BEflow1},  it holds  $\H_{t}f \in \LIP(X)$ with uniform 
Lipschitz bounds   for $t \in [0,T]$. It follows that, for a fixed $f \in \LIP(X)\cap L^\infty(X,\mm)$, the pair $(\H_{T-t} f , \Gamma(\H_{T-t}f))$ is  essentially bounded with values in $\R^2$ for $t \in [0, T]$. 
Therefore $\Psi$ is Lipschitz on the range of $(t, \H_{T-t} f , \Gamma(\H_{T-t}f))$  for $t \in [0, T]$, and thus 
$$ [0,T)\ni t \mapsto G(t):=\Psi(t, \H_{T-t} f , \Gamma(\H_{T-t}f)) \in { L^1(X,\mm)} \text{ is locally Lipschitz in $[0,T)$},$$
taking the local Lipschitz property of $t\mapsto \Gamma(\H_t f)$ as a $L^1(X,\mm)$-valued
map into account. In addition, \eqref{eq:der_Gamma}  and a standard chain rule provide the existence for $\LL^1$-a.e. $t_0\in (0,T)$ 
of the strong $L^1(X,\mm)$ derivative of $G$, given by
\begin{equation}\label{eq:dtPhiPf3}
G'(t_0)=\partial_t\Psi -\partial_u \Psi \Delta g- 2\partial_v \Psi\, \Gamma(g, \Delta g), 
\end{equation} 
with $g=\H_{T-t_0}f$ and $\Psi=\Psi(t_0,g,\Gamma g)$. Notice also that, by gradient contractivity,
$\|G\|_{L^\infty(X,\mm)}$ is bounded in $[0,T]$, hence $G$ is continuous in $(0,T]$ as a $L^2(X,\mm)$-valued map.
Using that $\H_t$ is self-adjoint, we can write $\Phi(t)=\int_X G(t)\H_t\varphi\,\d\mm$, and writing
$$
\Phi(t)-\Phi(t_0)=\int_X(G(t)-G(t_0))\H_{t_0}\varphi\,\d\mm+\int_X(\H_t\varphi-\H_{t_0} \varphi)G(t)\,\d\mm
$$
we can use the $L^2(X,\mm)$ continuity of $G$ to obtain the Leibniz rule
\begin{equation}\label{eq:dtPhiPf1}
\Phi'(t_0)=\int_X\bigl(\partial_t\Psi -\partial_u \Psi \Delta g- 2\partial_v \Psi\, \Gamma(g, \Delta g)\bigr)\H_{t_0}\varphi\,\d\mm+
\int_X G(t_0)\Delta\H_{t_0}\varphi\,\d\mm
\end{equation}
for $\LL^1$-a.e. $t_0\in (0,T)$.  This proves also the local Lipschitz property of $\Phi$ in $(0,T)$.
Continuity up to $t=0,T$ follows by \eqref{eq:BEflow1} and the assumption $f\in\LIP(X)$.
\\

\textbf{Step 2}. An intermediate approximation.
\\In this intermediate step, for $t\in (0,T)$ fixed, first we compute by chain rule the laplacian of $\Psi(t, g, \H_{\varepsilon} (\Gamma(g)))$ and then  show that  $\Delta\Psi(t, g, \H_{\varepsilon} (\Gamma(g)))$ admits  an explicit  (weak) limit  as $\varepsilon \downarrow 0$ which formally
can be considered as $\Delta^\ast\Psi(t, g,\Gamma(g))$; the expression of the (weak) limit will be then used in Step 3.
\\ Occasionally we shall
use the notation $\Psi_\epsilon$ for $\Psi(t,g,H_\varepsilon(\Gamma(g)))$, $\Psi$ for $\Psi(t,g,\Gamma(g))$
and analogous ones for $\Psi_u$, $\Psi_v$ and for second order partial derivatives.
\\
Since  we know that $g:=\H_{T-t}(f) \in \LIP(X)$, thus $\Gamma(g)\in L^{\infty}(X,\mm)$ and  $ \H_{\varepsilon}(\Gamma(g))\in L^{\infty}(X,\mm)\cap \LIP(X)\cap \dom(\Delta)$ for all $\varepsilon>0$. Therefore, by standard chain rule 
(see for instance \cite[(2.4)]{Savare2013}) we get that $\Psi(t,g,\H_{\varepsilon}(\Gamma(g))) \in \LIP(X)\cap \dom(\Delta)$ with
\begin{eqnarray}
\Delta \Psi_\varepsilon &=&
\partial_u \Psi(t,g,\H_{\varepsilon}(\Gamma(g))) \, \Delta g    +  \partial^2_u \Psi(t,g,\H_{\varepsilon}(\Gamma(g))) \, \Gamma(g) \nonumber\\
&&+2\, \partial_u \partial_v \Psi(t,g,\H_{\varepsilon}(\Gamma(g))) \, \Gamma(g, \H_{\varepsilon}( \Gamma(g)))+\partial^2_v\Psi (t,g,\H_{\varepsilon}(\Gamma(g)))\, \Gamma(\H_{\varepsilon}(\Gamma(g))) \nonumber\\
&&+\partial_v\Psi(t,g,\H_{\varepsilon}(\Gamma(g)))\Delta \H_{\varepsilon}(\Gamma(g)).\label{eq:DeltaPsieps}
\end{eqnarray}
We now pass to the limit as $\varepsilon \downarrow 0$ in the last formula. To this aim observe that by \eqref{eq:57} we know that $\Gamma(g)\in \V$, so that $\H_{\varepsilon} (\Gamma(g))\to \Gamma(g)$ in $\V$ and therefore, since $\Psi, \partial \Psi, \partial^{2}\Psi$ are Lipschitz on the range of $(t,g,\H_{\varepsilon}\Gamma(g))$ uniformly for $\varepsilon\in [0,1]$ it follows that
\begin{eqnarray}
&&\Psi(t,g,\H_{\varepsilon}(\Gamma(g))) \to \Psi(t,g,\Gamma(g)), \; \partial \Psi(t,g,\H_{\varepsilon}(\Gamma(g))) 
\to \partial \Psi(t,g,\Gamma(g)),  \nonumber\\
 &&   \partial^{2} \Psi(t,g,\H_{\varepsilon}(\Gamma(g))) \to \partial^{2} \Psi(t,g,\Gamma(g)) \; \text{ in  } \V, \label{eq:PsiepstoPsi}
\end{eqnarray}
where for brevity we wrote $\partial \Psi, \partial^{2}\Psi$ in place of $\partial_{u} \Psi, \partial_{v}\Psi, \partial_{u}^{2} \Psi, \partial_{v}^{2} \Psi, \partial_{u}\partial_{v}\Psi$.
Therefore the first two lines of \eqref{eq:DeltaPsieps} pass to the limit in $L^{1}(X,\mm)$ topology as $\varepsilon\downarrow 0$:
\begin{eqnarray}
&&\partial_u \Psi_\varepsilon\, \Delta g    +  \partial^2_u \Psi_\varepsilon \, \Gamma(g) 
+2\, \partial_u \partial_v \Psi_\varepsilon \, \Gamma(g, \H_{\varepsilon}( \Gamma(g)))
+\partial^2_v\Psi_\varepsilon\, \Gamma(\H_{\varepsilon}(\Gamma(g))) \nonumber\\
&\to&  \partial_u \Psi\, \Delta g    +  \partial^2_u \Psi \, \Gamma(g) +2\, \partial_u \partial_v \Psi \, \Gamma(g, \Gamma(g))
+\partial^2_v\Psi\, \Gamma(\Gamma(g)) \quad \text{ in } L^{1}(X,\mm). \label{eq:DeltaPsieps2lines}
\end{eqnarray}
Regarding the convergence of the  last line of \eqref{eq:DeltaPsieps}, observe that  $\Delta^{\star}\Gamma(g)\in \V'$, so that $\Delta \H_{\varepsilon} \Gamma(g) \mm  \to \Delta^{\star}\Gamma(g)$ weakly in $\V'$ topology as $\varepsilon\downarrow 0$. 
Therefore, the combination with \eqref{eq:PsiepstoPsi} gives
\begin{equation}\label{eq:DeltaPsieps3lines}
\langle \partial_v\Psi_\varepsilon, \Delta \H_{\varepsilon}(\Gamma(g)) \mm  \rangle_{\V, \V'} \to  
\langle \partial_v\Psi, \Delta^{\star} \Gamma(g)  \rangle_{\V, \V'}. 
\end{equation}
Combining  \eqref{eq:DeltaPsieps}, \eqref{eq:DeltaPsieps2lines} and  \eqref{eq:DeltaPsieps3lines} we conclude that for every $\varphi \in \LIP(X)\cap L^{\infty}(X,\mm)$ it holds
\begin{eqnarray}
\int_X\varphi\Delta \Psi_\varepsilon  \, \d \mm  &\to&  
\int_X \left[ \partial_u \Psi \, \Delta g    +  \partial^2_u \Psi\, \Gamma(g) +
2\, \partial_u \partial_v \Psi \, \Gamma(g, \Gamma(g))+\partial^2_v\Psi \, \Gamma(\Gamma(g))  \right]  \, \varphi \, \d \mm \nonumber\\
&& + \langle \varphi\;  \partial_v\Psi, \Delta^{\star} \Gamma(g)  \rangle_{\V, \V'}, \quad \text{ as } \varepsilon\downarrow 0.  \label{eq:ConvergenceEps0}
\end{eqnarray}

\textbf{Step 3}. Formula \eqref{eq:dtPhi} holds.\\ 
Set
$$
\Phi_{\varepsilon}(t):=\int_{X} \Psi(t, \H_{T-t}(f), \H_{\varepsilon}(\Gamma(\H_{T-t} f))) \; \H_{t}(\varphi) \, \d \mm. 
$$
By analogous arguments of Step 1, we get that $\Phi_{\varepsilon}$ is locally Lipschitz on $(0,T)$ with 
\begin{eqnarray}\label{eq:dtPhiepsPf1}
\Phi'_{\varepsilon}(t_0)&=&\int_X\big[ \partial_t\Psi(t_{0}, g, \H_{\varepsilon} (\Gamma(g)))  
 -\partial_u \Psi(t_{0}, g , \H_{\varepsilon}( \Gamma(g))) \,   \Delta g \big]  \; \H_{t_0}\varphi\  \,\d\mm \nonumber\\
 &&-2  \int_{X} \partial_v \Psi(t_{0}, g , \H_{\varepsilon}( \Gamma(g)))  \, \H_{\varepsilon} (\Gamma(g, \Delta g)) \; \H_{t_0}\varphi\,\d\mm  \nonumber\\
 && + \int_X \Psi(t_{0}, g , \H_{\varepsilon}( \Gamma(g))) \, \Delta\H_{t_0}\varphi\,\d\mm,
\end{eqnarray}
for $\LL^1$-a.e. $t_0\in (0,T)$. Combining  \eqref{eq:dtPhiPf1}, \eqref{eq:PsiepstoPsi} and \eqref{eq:dtPhiepsPf1} we infer that, given a sequence $\varepsilon_{n}\downarrow 0$ it holds 
\begin{equation}\label{eq:Phi'espPhi'}
\Phi_{\varepsilon_{n}}'(t_{0}) \to \Phi'(t_{0}),\quad 	\text{for $\LL^1$-a.e. $t_0\in (0,T)$}. 
\end{equation}
By Step 2 we know that $\Psi(t,g,\H_{\varepsilon}(\Gamma(g))) \in \LIP(X)\cap \dom(\Delta)$ with $\Delta \Psi(t,g,\H_{\varepsilon}(\Gamma(g)))$ given by \eqref{eq:DeltaPsieps};  therefore we can integrate by parts the laplacian in the last integral in \eqref{eq:dtPhiepsPf1} and use  \eqref{eq:PsiepstoPsi}, \eqref{eq:ConvergenceEps0} in order to pass to the limit as $\varepsilon_{n} \downarrow 0$:
\begin{eqnarray}
\Phi'(t_{0})&=&\lim_{n\to \infty }\Phi'_{\varepsilon_{n}}(t_0) \nonumber\\
&=&\int_X\bigl(\partial_t\Psi -\partial_u \Psi \Delta g- 2\partial_v \Psi\, \Gamma(g, \Delta g)\bigr)\H_{t_0}\varphi\,\d\mm  \nonumber\\
&&+\int_X \bigl(\partial_u \Psi \, \Delta g    +  \partial^2_u \Psi \, \Gamma(g)+2\, \partial_u \partial_v \Psi \, \Gamma(g, \Gamma(g))+\partial^2_v\Psi\, \Gamma(\Gamma(g)) \bigr)  \H_{t_{0}}\varphi\,\d\mm  \nonumber\\
&&+\int_X \partial_v\, \Psi \, \H_{t_0}\varphi\,\d\Delta^\star\Gamma(g), \quad 	\text{for $\LL^1$-a.e. $t_0\in (0,T)$},  \label{eq:limdtPhieps}
\end{eqnarray}
where in the first line we used \eqref{eq:Phi'espPhi'}  and we wrote simply $\Psi$ in place of $\Psi(t_{0},g,\Gamma(g))$.

Recall that, thanks to Theorem~\ref{thm:Gamma2Impr}, 
$$\frac 12 \Delta^\star \Gamma(g)=\Big(\Gamma(g,{\Delta g})+K\Gamma(g)\Big)\mm +  \Gamma^\perp_{2,K} [g]+ 
\gamma_{2,K}[g] \mm,$$ 
where $\Gamma^\perp_{2,K} [g]$ is nonnegative and singular with respect to $\mm$ and  $\gamma_{2,K}[g]\mm$ is the nonnegative absolutely continuous part. Since $\H_{t_0}\varphi$ is nonnegative by the minimum principle of the heat flow and  $\partial_v \Psi\geq 0$ by assumption,  
neglecting the contribution of the singular part the thesis follows.
\end{proof}

\textbf{Proof of Theorem~\ref{thm:LocBob}}

\textbf{Step 1}. It is enough to show the validity of \eqref{eq:LocBob} for every $f \in \LIP(X)$ taking values in $[\varepsilon, 1-\varepsilon]$, for some $\varepsilon\in(0,1/2)$.

Let  $f \in \LIP(X)$ taking values in $[0, 1]$ and for every $\varepsilon\in (0,1/2)$ define $f_{\varepsilon}$ with values in  $[\varepsilon, 1-\varepsilon]$ by truncation as 
 $$f_{\varepsilon}:=\max(\min(f, 1-\varepsilon), \varepsilon).$$
In this first step we show that, assuming the validity of   \eqref{eq:LocBob}  for the truncated $f_{\varepsilon}$, then we can pass into the limit and get the validity of \eqref{eq:LocBob} also for $f$. 
To this aim, first observe that since the truncation satisfies $\Gamma( f_{\varepsilon})\leq \Gamma(f)$ which is essentially bounded 
by the assumption $f\in \LIP(X)$, and since $\cI$ takes values into $\bigl[0, \frac{1}{\sqrt{2 \pi}} \bigr]$, then there exists $C>0$ 
depending on $f$ but not on $\varepsilon$ such that
\begin{equation}\label{eq:unifBound}
 \sqrt{\cI^{2}(f_{\varepsilon})+c_{\alpha}(t)  \Gamma(f_{\varepsilon})}   \leq C \quad \mm\text{-a.e.}, \; \forall \varepsilon\in (0,1/2).  
\end{equation}
Moreover it is readily seen that  $f_{\varepsilon}\to f$ in $\V$-topology  and then, since $\H_{t}(\cdot)$ is continuous as map from $\V$ to $\V$,  also $\H_{t }f_{\varepsilon}\to \H_{t} f$ in $\V$-topology, as $\varepsilon\to 0$.
In particular, for every fixed $t\geq 0$, there exists  a sequence  $\varepsilon_{n}\downarrow 0$ such that 
\begin{equation}\label{eq:fve0}
f_{\varepsilon_{n}}\to f, \quad \H_{t} f_{\varepsilon_{n}}\to \H_{t} f, \quad \Gamma(f_{\varepsilon_{n}})\to \Gamma(f),  \quad \Gamma(\H_{t}f_{\varepsilon_{n}})\to \Gamma(\H_{t}f),  \quad \mm\text{-a.e.}
\end{equation}
as $n\to\infty$.
We can then pass into the limit $\mm$-a.e. as $n\to \infty$ in the left hand side of \eqref{eq:LocBob}. In order to pass into  the limit $\mm$-a.e. also in the right hand side of \eqref{eq:LocBob} observe that thanks to \eqref{eq:unifBound} and \eqref{eq:fve0} we can apply the Dominated Convergence Theorem to infer that 
$$\sqrt{\cI^{2}(f_{\varepsilon_{n}})+c_{\alpha}(t)  \Gamma(f_{\varepsilon_{n}})} \to \sqrt{\cI^{2}(f)+c_{\alpha}(t)  \Gamma(f)} \quad \text{ in $L^{2}(X,\mm)$-strong topology, as }n\to \infty.  $$
Using again the continuity of the heat flow $\H_{t}(\cdot)$ as map from $L^{2}(X,\mm)$ to $L^{2}(X,\mm)$ we conclude that, 
possibly along a subsequence, it holds
$$\H_{t}\left(\sqrt{\cI^{2}(f_{\varepsilon_{n}})+c_{\alpha}(t)  \Gamma(f_{\varepsilon_{n}})}\right) \to \H_{t}\left( \sqrt{\cI^{2}(f)+c_{\alpha}(t)  \Gamma(f)} \right) \quad \mm\text{-a.e. as } n\to \infty. $$
Therefore we can pass into the limit  $\mm$-a.e. also in the right hand side of \eqref{eq:LocBob} and conclude that \eqref{eq:LocBob} holds also for $f$.
\\

\textbf{Step 2}. Explicit computation of $\zeta$.
\\Thanks to Step 1 we can assume $f$ to take values into $[\varepsilon, 1-\varepsilon]$, for some $\varepsilon \in (0,1/2)$. Therefore by maximum principle also $\H_{t} f$ will take values into $[\varepsilon, 1-\varepsilon]$, and thus $\cI(\H_{T-t} f)$ will take values into 
$\bigl[\delta, \frac{1}{\sqrt{2\pi}}-\delta\bigr]$ for some $\delta\in (0,1/2\sqrt{2\pi})$ 
depending on $\varepsilon$. It follows that, in the range of $(t, H_{T-t} f,  \Gamma(\H_{T-t} f))$, the function
\begin{equation}\label{eq:PsiSqrt}
\Psi(t,u,v):=\sqrt{\cI^{2}(u)+ c_{\alpha}(t) v}, \quad (t,u,v)\in (0,T)\times \left[\delta, \frac{1}{\sqrt{2\pi}}-\delta \right] \times [0,\infty)
\end{equation}
is of class $C^4$ and we can then apply Lemma \ref{lem:compLocBobkov}. Writing $\cI, \cI'$ in place of $\cI(u), \cI'(u)$ for brevity, it is immediate to check that
$$
\Psi \, \partial_t\Psi =\frac{c_{\alpha}'}{2} v, \qquad \Psi \, \partial_u \Psi= \cI \, \cI', \qquad \Psi \, \partial_v \Psi=\frac{c_{\alpha}}{2},
$$
and, by differentiating once more and using that $\cI\, \cI''=-1$,
$$
\Psi^3 \, \partial_u^2\Psi=-\cI^2 \, \cI'^{2}+ \Psi^2(\cI'^{2}-1), \qquad 
\Psi^3 \partial_u\partial_v\Psi=-\frac{c_{\alpha}}{2} \cI\,\cI', \qquad 
\Psi^3 \partial_v^2\Psi=-\frac{c_{\alpha}^2}{4}.
$$
Therefore, the corresponding function $\zeta$ defined in \eqref{eq:defzeta}, satisfies $\mm$-a.e. 
\begin{eqnarray}
\Psi^{3} \zeta&=& \Psi^{2} \frac{c_{\alpha}'(t)}{2} \Gamma(g)-\cI^{2}(g) \, \cI'^{2}(g) \Gamma(g)+ \Psi^{2}(\cI'^{2}(g)-1)\Gamma(g) \nonumber \\
&&-c_{\alpha}(t) \cI(g)\,\cI'(g)   \,   \Gamma(g, \Gamma(g)) -\frac{c_{\alpha}^{2}(t)}{4} \, \Gamma(\Gamma(g))  + K \Psi^{2}\, c_{\alpha}(t)  \,\Gamma(g) +   \Psi^{2} c_{\alpha}(t)  \gamma_{2,K}[g], \nonumber
\end{eqnarray}
where we set $g=\H_{T-t} f$ and  $\Psi=\Psi(t,g,\Gamma(g))$ for brevity. Since $\Psi^2(t,u,v)=\cI^2(u)+c_\alpha(t)v$, it 
follows that the following equalities hold $\mm$-a.e.:
\begin{eqnarray}
\Psi^{3} \zeta&=& \left[\cI^{2}(g)+c_{\alpha}(t) \Gamma(g) \right] \frac{c_{\alpha}'(t)}{2} \Gamma(g)-\cI^{2}(g) \, \cI'^{2}(g) \Gamma(g)+  \left[\cI^{2}(g)+c_{\alpha}(t) \Gamma(g) \right] (\cI'^{2}(g)-1)\Gamma(g) \nonumber \\
&&-c_{\alpha}(t) \cI(g)\,\cI'(g)   \,   \Gamma(g, \Gamma(g)) -\frac{c_{\alpha}^{2}(t)}{4} \, \Gamma(\Gamma(g))  + K  \left[\cI^{2}(g)+c_{\alpha}(t) \Gamma(g) \right]\, c_{\alpha}(t)  \,\Gamma(g)  \nonumber \\
&& +   \left[\cI^2(g)+c_{\alpha}(t) \Gamma(g) \right]\, c_{\alpha}(t)  \,\gamma_{2,K}[g]  \nonumber\\
&=& c_{\alpha}(t) \Gamma^2(g) \left(\frac{c_{\alpha}'(t)}{2}-1 \right) +c_{\alpha}^{2}(t) \left[ - \frac{1}{4} \Gamma(\Gamma(g)) + \Gamma(g)  \,\gamma_{2,K}[g]  \right]+ \cI^2(g)  \Gamma(g) \left(\frac{c_{\alpha}'(t)}{2}-1 \right)  \nonumber \\
&& - c_{\alpha}(t) \cI'(g)\, \cI(g) \, \Gamma(g, \Gamma(g)) + c_{\alpha}(t) \, \cI'^{2}(g)\, \Gamma^2(g) +   K  \left[\cI^{2}(g)+c_{\alpha}(t) \Gamma(g) \right] c_{\alpha}(t)  \,\Gamma(g)  \nonumber \\
&& +   \cI^{2}(g) \, c_{\alpha}(t)  \,\gamma_{2,K}[g].\nonumber
\end{eqnarray}
By the very definition \eqref{eq:defcalpha} of $c_{\alpha}(t)$, we have that $\frac{c_{\alpha}'(t)}{2}-1=-K c_{\alpha}(t)$ so that the
above formula for $\zeta$ simplifies to
\begin{eqnarray}
\Psi^3 \zeta&=&c_\alpha^2(t)\left[ -\frac{1}{4} \Gamma(\Gamma(g))+ \Gamma(g)  \,\gamma_{2,K}[g] \right]  \nonumber \\
&&+c_{\alpha}(t) \left[ -  \cI'(g)\, \cI(g) \, \Gamma(g, \Gamma(g)) + \cI'^{2}(g)\, \Gamma^2(g)+   \cI^{2}(g)  \,\gamma_{2,K}[g] \right] \; \; \; \mm\text{-a.e. }. \label{eq:Psi3z}
\end{eqnarray}

\textbf{Step 3}. Conclusion of the proof. 
\\First of all observe that thanks to Lemma~\ref{lem:regHF} we can  apply   Theorem~\ref{thm:Gamma2Impr} (which, let us stress,  is implied by the curvature condition $\RCD(K,\infty)$) to $g=\H_{T-t} f$ and get 
\begin{equation}\label{eq:GammaGamma}
\Gamma\left(\Gamma (g)\right) \le 4 \Gamma(g)\; \gamma_{2,K}  [g], \quad \mm\text{-a.e. }.
\end{equation}
Since $\Psi=\Psi(t,\H_{T-t}f, \Gamma(\H_{T-t} f))>0$ as we are assuming $f$ to take values into $[\varepsilon,1-\varepsilon]$, the combination of \eqref{eq:dtPhi},  \eqref{eq:Psi3z} and \eqref{eq:GammaGamma}  gives
\begin{equation}\label{eq:QuadForm}
\frac{d}{dt}_{|t=t_{0}} \Phi(t) \geq  \int_{\{\Gamma(g)>0\}} \frac{c_{\alpha}(t)}{\Psi^{3}}  \left[ \cI^{2}(g) \frac{\Gamma(\Gamma(g))}{4 \Gamma(g)}  -  \cI'(g)\, \cI(g) \, \Gamma(g, \Gamma(g)) + \cI'^{2}(g)\, \Gamma^2(g) \right]\; \H_{t_{0}} \varphi \, \d \mm, 
\end{equation}
for all $\varphi\in L^{\infty}(X,\mm),$ with  $\varphi\geq 0$  $\mm$-a.e. in $X$. 
Notice that the restriction of the region of integration
to $\{\Gamma(g)>0\}$ is possible because the nonnegativity of $\gamma_{2,K}[g]$ and $\Gamma(\Gamma(g))=0$ $\mm$-a.e.
on $\{\Gamma(g)=0\}$ imply that $\zeta$ is nonnegative $\mm$-a.e. on $\{\Gamma(g)=0\}$.
%
% where we set, by definition,  $\frac{\Gamma(\Gamma(g))}{4 \Gamma(g)}=0$ on $\{ \Gamma(g)=0\}$: this is perfectly compatible with  \eqref{eq:Psi3z} and \eqref{eq:GammaGamma} since, by the very definition  \eqref{eq:50}, one has $ \gamma_{2,K}[g]=0$ on $\mm\llcorner \{ \Gamma(g)=0\}$-a.e. thanks to the locality of  $\Delta^{\star}$ and $\Gamma$.
\\Now observe that the right hand side of \eqref{eq:QuadForm} is a quadratic form in $\cI(g), \cI'(g)$ which is positive  definite since 
$$\frac{\Gamma(\Gamma(g))}{4 \Gamma(g)}\geq 0,\quad
\Gamma(g)^2 \geq 0,\quad\text{and}\quad  \Gamma(g, \Gamma(g))\leq 
\Gamma^{{ 1/2}}(g)\; \Gamma^{{1/2}}(\Gamma(g)), \quad \mm\text{-a.e. }.  $$
We conclude that for $\LL^{1}$-a.e. $t_{0}\in (0,T)$ it holds $\frac{d}{dt}_{|t=t_{0}} \Phi(t) \geq 0$. Since by 
Lemma~\ref{lem:compLocBobkov} we know that $\Phi$ is continuous in $[0,T]$ and locally Lipschitz in $(0,T)$, we can integrate to get
\begin{eqnarray}
0&\leq&\int_0^T \Phi'(t) \, \d t=\Phi(T)-\Phi(0) \nonumber \\
&=& \int_{X} \left[ \H_{T} \left( \sqrt{\cI^{2}(f)+c_{\alpha}(T)  \Gamma(f)} \right)-  \sqrt{\cI^{2}(\H_{T}f)+ \alpha \Gamma(\H_{T}f)}   \right] \; \varphi \, \d \mm, 
\end{eqnarray}
for all $\varphi\in L^{\infty}(X,\mm),$ with  $\varphi\geq 0$  $\mm$-a.e., which completes the proof of Theorem~\ref{thm:LocBob}.
\hfil$\Box$

\section{Bobkov  and Gaussian isoperimetric inequalities for  $\RCD(K,\infty)$ probability spaces, $K>0$}

In this section we specialize the local Bobkov Inequalitiy  \eqref{eq:LocBob} to the case $K>0$, 
and use it to prove the  Gaussian Isoperimetric Inequality.

\begin{theorem}[Bobkov Inequality]\label{thm:Bob}
Let $(X,\sfd,\mm)$ be an $\RCD(K,\infty)$ probability space for some $K>0$. Then, for every $f \in \V$ with values in $[0,1]$, it holds
\begin{equation}\label{eq:Bob}
\sqrt{K} \; \cI\left(\int_{X} f \, \d\mm \right) \leq \int_{X} \sqrt{K \cI^{2}(f)+\Gamma(f)} \, \d\mm.
\end{equation}
\end{theorem}

\begin{proof}
Inequality \eqref{eq:Bob} for $f \in \LIP(X)$ taking values in $[0,1]$ follows  by the combination of the ergodicity of the Heat flow and the local Bobkov inequality. More precisely,  set   $\alpha=1/K$ (so that $c_{\alpha}(t)=1/K$) and let $t\to + \infty$ in  
$$
\int_X \cI(\H_{t}f)\,\d\mm\leq \int_X\sqrt{\cI^{2}(f)+\frac 1K  \Gamma(f)}  \,\d\mm, \quad \mm\text{-a.e.},
$$
derived from \eqref{eq:LocBob} dropping one nonnegative term and integrating both sides;  
\eqref{eq:Bob} follows then by applying the ergodicity property \eqref{eq:ergod}.
\\The validity of  \eqref{eq:Bob}  for all $f \in \V$ taking values in $[0,1]$ follows by the density of Lipschitz functions in $\V$, see  \cite[Proposition 4.10]{Ambrosio-Gigli-Savare11b}. 
\end{proof}

Applying \eqref{eq:Bob} to functions approximating the characteristic function of a set $E\subset X$ we get the desired Gaussian isoperimetric inequality.
More precisely, recall the standard definition of $BV(X,\sfd,\mm)\subset L^1(X,\mm)$ space and of total variation $|\rmD f|(X)$
$$
|\rmD f|(X):=\inf\left\{\liminf_{n\to\infty}\int_X {\rm lip\,} f_n\,\d\mm:\
f_n\in\LIP(X),\,\,\lim_{n\to\infty}\int_X|f_n-f|\,\d\mm=0\right\},
$$
we can pass to the limit in the inequality (derived from \eqref{eq:Bob} from the subadditivity of the square root
and the inequality $\sqrt{\Gamma(f)}\leq{\rm lip\,}f$ $\mm$-a.e.)
$$
\sqrt{K} \; \cI\left(\int_X f \, \d\mm \right) \leq \int_X \bigl(\sqrt{K} \,\cI(f)+{\rm lip\,} f\bigr) \, \d\mm
$$
for all $f\in\LIP(X)$ with values in $[0,1]$ to get
$$
\sqrt{K} \; \cI\left(\int_X f \, \d\mm \right) \leq \sqrt{K}\int_X  \cI(f)\,\d\mm+|\rmD f|(X)
\qquad\forall f\in BV(X,\sfd,\mm).
$$
In particular, by applying this to characteristic functions $f=\chi_E$, since $\cI(0)=\cI(1)=0$,
we obtain the Gaussian isoperimetric inequality with the perimeter ${\mathcal P}(E):=|\rmD \chi_E|(X)$:

\begin{theorem}\label{thm:IsopIneq}
Let $(X,\sfd,\mm)$ be an $\RCD(K,\infty)$ probability space for some $K>0$. Then, for every Borel subset $E\subset X$ it holds
\begin{equation}\label{eq:Isop}
{\mathcal P}(E)\geq \sqrt{K} \,\cI(\mm(E)).
\end{equation}
\end{theorem}

\def\cprime{$'$}

\end{document}